\documentclass[12pt,draftcls,onecolumn]{IEEEtran}
\usepackage{graphicx}
\usepackage{amssymb}
\usepackage{amsmath}
\usepackage{amsfonts}
\usepackage{comment}

\title{Polynomial stochastic games \\ via sum of squares optimization}

\author{Parikshit Shah and Pablo A. Parrilo  \\
Department of Electrical Engineering and Computer Science \\
Massachusetts Institute of Technology, Cambridge, MA 02139
\thanks{This research was funded in part by AFOSR MURI subawards 2003-07688-1 and 102-1080673.} }

\def\T={\buildrel {\scriptscriptstyle\triangle} \over =}

\def\Ex{{\mathbf{E}}}
\def\p{{^{\prime}}}
\def\mc{\mathcal}
\def\mb{\mathbf}
\def\mh{{\mathcal{H}(\bar{\nu})(s)}}
\def\t{{^{T}}}
\def\xib{{\bar{\xi}}}
\def\mxi{{\mathcal{H}(\xib(s))}}
\def \val{\mathrm{val}}

\newtheorem{theorem}{Theorem}%[section] %if you include section, it will show theorem 2.4 etc,
\newtheorem{definition}{\indent Definition}
\newtheorem{lemma}{Lemma}

\newenvironment{proof}[1][Proof]{\begin{trivlist}
\item[\hskip \labelsep {\bfseries #1}]}{\end{trivlist}}

\newenvironment{remark}[1][Remark]{\begin{trivlist}
\item[\hskip \labelsep {\bfseries #1}]}{\end{trivlist}}

\begin{document}

\pagenumbering{arabic}
 \maketitle %\thispagestyle{empty} %\pagestyle{empty}

%%%%%%%%%%%%%%%%%%%%%%%%%%%%%%%%%%%%%%%%%%%%%%%%%%%%%%%%%%%%%%%%%%%%%%%%%%%%%%%%
\begin{abstract}
Stochastic games are an important class of problems that generalize
Markov decision processes to game theoretic scenarios. We consider
finite state two-player zero-sum stochastic games over an infinite
time horizon with discounted rewards. The players are assumed to have
infinite strategy spaces and the payoffs are assumed to be
polynomials. In this paper we restrict our attention to a special
class of games for which the \emph{single-controller assumption}
holds. It is shown that minimax equilibria and optimal strategies for
such games may be obtained via semidefinite programming.
\end{abstract}

%%%%%%%%%%%%%%%%%%%%%%%%%%%%%%%%%%%%%%%%%%%%%%%%%%%%%%%%%%%%%%%%%%%%%%%%%%%%%%%%

\section{Introduction}

Markov decision processes (MDPs) are very widely used system modeling
tools where a single agent attempts to make optimal decisions at each
stage of a multi-stage process so as to optimize some reward or payoff
\cite{Bertsekas}. Game theory is a system modeling paradigm that
allows one to model problems where several (possibly adversarial)
decision makers make individual decisions to optimize their own payoff
\cite{Fudenberg}. In this paper we study \emph{stochastic games}
\cite{Koos}, a framework that combines the modeling power of MDPs and
games. Stochastic games may be viewed as \emph{competitive MDPs} where
several decision makers make decisions at each stage to maximize their
own reward. Each state of a stochastic game is a simple game, but the
decisions made by the players affect not only their current payoff,
but also the transition to the next state. 

Notions of solutions in games have been extensively studied, and are
very well understood. The most popular notion of a solution in game
theory is that of a \emph{Nash equilibrium}. While these equilibria
are hard to compute in general, in certain cases they may be computed
efficiently. For games involving two players and finite action spaces,
mixed strategy minimax equilibria always exist (see, e.g.,
\cite{Fudenberg}). These minimax saddle points correspond to the
well-known notion of a Nash equilibrium. From a computational
standpoint such games are considered tractable because Nash equilibria
may be computed efficiently via linear programming. Stochastic games
were introduced by Shapley \cite{Shapley} in 1953. In his paper, he
showed that the notion of a minimax equilibrium may be extended to
stochastic games with finite state spaces and strategy sets. He also
proposed a value iteration-like algorithm to compute the equilibria.
In 1981 Parthasarathy and Raghavan \cite{Partha,Koos} studied single
controller games. Single controller games are games where the
probabilities of transitions are controlled by the action of only one
player. They showed that stochastic games satisfying this property
could be solved efficiently via linear programming (thus proving that
such problems with rational data could be computed in a finite number
of steps).

While computational techniques for finite games are reasonably well
understood, there has been some recent interest in the class of
\emph{infinite games}; see \cite{Par,Stein} and the references
therein. In this important class, players have access to an infinite
number of pure strategies, and the players are allowed to randomize
over these choices. In a recent paper \cite{Par}, Parrilo describes a
technique to solve two-player, zero-sum infinite games with polynomial
payoffs via semidefinite programming. It is natural to wonder whether
the techniques from finite stochastic games can be extended to
infinite stochastic games (i.e. finite state stochastic games where
players have access to infinitely many pure strategies). In
particular, since finite, single-controller, zero-sum games can be
solved via linear programming, can similar infinite stochastic games
be solved via semidefinite programming? The answer is affirmative, and
this paper focuses on establishing this result.

The main contribution of this paper is to provide a computationally
efficient, finite dimensional characterization of the solution of
single-controller polynomial stochastic games. For this, we extend the
linear programming formulation that solves the finite action
single-controller stochastic game (i.e., under assumption (SC) below),
to an infinite dimensional optimization problem when the actions are
uncountably infinite. We furthermore establish the following
properties of this infinite dimensional optimization problem:
\begin{enumerate}
\item Its optimal solutions correspond to minimax equilibria.
\item The problem can be solved efficiently by semidefinite programming.
\end{enumerate}
Section~\ref{sec:main} of this paper provides a formal description of
the problem and introduces the basic notation used in the paper. We
show that for two-player zero-sum polynomial stochastic games,
equilibria exist and that the corresponding equilibrium value vector
is unique. (This proof is essentially an adaptation of the original
proof by Shapley in \cite{Shapley} for finite stochastic games). In
Section~\ref{sec:main} we also briefly review some elegant results
about polynomial nonnegativity, moment sequences of nonnegative
measures, and their connection to semidefinite programming. In
Section~\ref{sec:finite}, we briefly review the linear programming
approach to finite stochastic games. Section~\ref{sec:infinite} states
and proves the main result of this paper. In Section~\ref{sec:example}
we present an example of a two-player, two-state stochastic game, and
compute the equilibria via semidefinite programming. Finally, in
Section~\ref{sec:conclusions} we state some natural extensions of this
problem, conclusions, and directions of future research.

%%%%%%%%%%%%%%%%%%%%%%%%%%%%%%%%%%%%%%%%%%%%%%%%%%%%%%%%%%%%%%%%%%%%%%%%%%%%%%%%
\section{Problem description}
\label{sec:main}

\subsection{Stochastic games}
We consider the problem of solving two-player zero-sum stochastic
games via mathematical programming. The game consists of finitely
many states with two adversarial players that make simultaneous
decisions. Each player receives a payoff that depends on the actions
of both players and the state (i.e. each state can be thought of as
a particular zero-sum game). The transitions between the states are
random (as in a finite state Markov decision process), and the
transition probabilities in general depend on the actions of the
players and the current state. The process runs over an infinite
horizon. Player $1$ attempts to maximize his reward over the horizon
(via a discounted accumulation of the rewards at each stage) while
player $2$ tries to minimize his payoff to player $1$. If
$(a_{1}^{1},a_{1}^{2},\ldots)$ and $(a_{2}^{1},a_{2}^{2},\ldots)$
are sequences of actions chosen by players $1$ and $2$ resulting in
a sequence of states $(s_{1},s_{2},\ldots)$ respectively, then the
reward of player $1$ is given by:
$$
\sum_{k=1}^{\infty}\beta^{k}r(s_{k},a_{1}^{k},a_{2}^{k}).
$$The game is completely
defined via the specification of the following data:
\begin{enumerate}
\item The (finite) state space $\mathcal{S}=\{1,\ldots,S\}$.
\item The sets of actions for players $1$ and $2$ given by
$A_{1}$ and $A_{2}$.
\item The payoff function, denoted by $r(s,a_{1},a_{2})$, for a given set of state $s$ and actions $a_{1}$ and $a_{2}$
(of players $1$ and $2$).
\item The probability transition matrix $p(s^{\prime};s,a_{1},a_{2})$ which provides the conditional probability of
transition from state $s$ to $s^{\prime}$ given players' actions.
\item The discount factor $\beta$, where $ 0 \leq \beta < 1$.
\end{enumerate}

To fix ideas, consider the following example of a two-state
stochastic game (i.e. $\mc{S}=\left\{ 1,2 \right \}$). The action
spaces of the two players are $A_{1}=A_{2}=[0,1]$. The payoff
function in state $1$ is $r(1,a_1,a_2)=r_1(a_1,a_2)$ and the payoff
function in state $2$ is given by $r(2,a_1, a_2)=r_2(a_1,a_2)$. Both
are assumed to be polynomials in $a_1$ and $a_2$. The probability
transition matrix is:
$$
P=\left[ \begin{array}{cc} p_{11}(a_1,a_2) & p_{12}(a_1,a_2) \\
p_{21}(a_1,a_2) & p_{22}(a_1, a_2)
\end{array}
\right].
$$ Every entry in this matrix is assumed to be a polynomial in $a_1$
and $a_2$. This stochastic game can be depicted graphically as shown
in Fig.~\ref{fig:game}. We will return to a specific instance of this
example in Section~\ref{sec:example}, where we explicitly solve for
the equilibrium strategies of the two players.
\begin{figure}
\begin{center}
\includegraphics{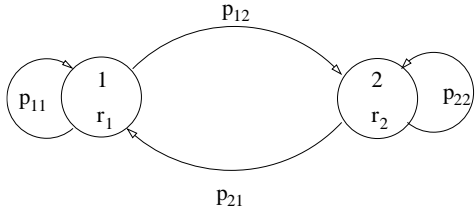}
\end{center}
\caption{A two state stochastic game. The payoff functions
associated to the states are denoted by $r_1$ and $r_2$. The edges
are marked by the corresponding state transition probabilities. }
\label{fig:game}
\end{figure}

Through most of this paper (except Section~\ref{sec:existence}) we make the
following important assumption
about the probability transition matrix: \\ \\
{\bf Assumption SC} \\
The probability transition to state $s\p$ conditioned upon the
current state being $s$ depends only on $s$, $s \p$, and the action
$a_{1}$ of player $1$ for every $s$ and $s\p$. This probability is
\emph{independent of the action of player $2$}. Thus,
$p(s\p;s,a_{1},a_{2})=p(s\p;s,a_{1})$. This is known as the
\emph{single-controller assumption}.\\ 

In this paper we will mostly (except briefly, in
Section~\ref{sec:finite} where finite strategy spaces are considered)
be concerned with the case where the action spaces $A_{1}$ and $A_{2}$
of the two players are uncountably infinite sets. For the sake of
simplicity we will often consider the case where
$A_{1}=A_{2}=[0,1]\in\mathbb{R}$. The results easily generalize to the
case where the strategy sets are finite unions of arbitrary intervals
of the real line. For the sake of simplicity, we also assume that the
action sets are the same for each state, though this assumption may be
relaxed. We will denote by $a_{1}$ and $a_{2}$, the actual actions
chosen by players $1$ and $2$ from their respective action spaces. The
payoff function is assumed to be a polynomial in the variables $a_{1}$
and $a_{2}$ with real coefficients: $$ r(s,a_{1},
a_{2})=\sum_{i=1}^{d_{1}}\sum_{j=1}^{d_{2}}r_{ij}(s)a_{1}^{i}a_{2}^{j}.
$$
Finally, we assume that the transition
probability $p(s\p;s,a_{1})$ is a polynomial in the action $a_{1}$.

The decision process runs over an infinite horizon,  thus it
is natural to restrict one's attention to stationary strategies for
each player, i.e. strategies that depend only on the state of the
process and not on time. Moreover, since the process involves two
adversarial decision makers, it is also natural to look for
randomized strategies (or mixed strategies) rather than pure
strategies so as to recover the notion of a minimax equilibrium. A
\emph{mixed} strategy for player $1$ is a finite set of probability
measures $\mathbf{\mu}=\left[ \mu(1), \ldots, \mu(S) \right]$
supported on the action set $A_{1}$. Each probability measure
corresponds to a randomized strategy for player $1$ in some
particular state, for example $\mu(k)$ corresponds to the randomized
strategy that player $1$ would use when in state $k$. Similarly,
player $2$'s strategy will be represented by $\mathbf{\nu}=\left[
\nu(1), \ldots, \nu(S) \right]$. (A word on notation: Throughout the
paper, indices in parentheses will be used to denote the state. Bold
letters will be used indicate vectorization with respect to the
state, i.e., collection of objects corresponding to different states
into a vector with the $i^{th}$ entry corresponding to state $i$.
The Greek letters $\xi$, $\mu$, $\nu$ will be used to denote
measures. Subscripts on these Greek letters will be used to denote
moments of the measures. A bar over a greek letter indicates a (finite) moment sequence (the length of the sequence being clear from the context). For example $\xi_{j}(i)$ denotes the
$j^{th}$ moment
of the measure $\xi$ corresponding to state $i$, and $\bar{\xi}(i)=[\xi_{0}(i), \ldots, \xi_{n}(i)]$).

A strategy $\mathbf{\mu}$ leads to a probability transition matrix
$P(\mathbf{\mu})$ such that
$P_{ij}(\mathbf{\mu})=\int_{A_{1}}p(j;i,a_{1})d\mu(i)$. Thus, once
player $1$ fixes a strategy $\mathbf{\mu}$, the probability
transition matrix is fixed, and can be obtained by integrating each
entry in the matrix with respect to the measure $\mathbf{\mu}$.
(Since the entries are polynomials, upon integration, these entries
depend affinely on the moments $\mu(i)$). Given strategies
$\mathbf{\mu}$ and $\mathbf{\nu}$, the expected reward collected by
player $1$ in some stage $s$ is given by:
$$r(s,\mu(s),\nu(s))= \int_{A_{1}} \int_{A_{2}} r(s,a_{1},a_{2})d\mu(s)d\nu(s).$$ The reward collected
over the infinite horizon (for fixed strategies $\mu(s)$ and
$\nu(s)$) starting at state $s$, $v_{\beta}(s,\mu(s),\nu(s))$, is
given by the system of equations:
\begin{equation*}
\begin{array}{l}
v_{\beta}(s,\mu(s),\nu(s))=r(s,\mu(s),\nu(s))+
\\ \text{\indent \indent }\beta\sum_{s \p \in \mc{S}} \left(
\int_{A_{1}}p(s^{\prime};s,a_{1})d\mu(s)
\right)v_{\beta}(s^{\prime},\mu(s \p), \nu(s \p)) \text{\indent}
\forall s.
\end{array}
\end{equation*}
Vectorizing $v_{\beta}(s,\mu(s),\nu(s))$, we obtain
$$\mb{v}_{\beta}(\mu,\nu)=(I-\beta P(\mathbf{\mu}))^{-1}\mb{r}(\mathbf{\mu},\mathbf{\nu}),$$
where $\mb{r}(\mathbf{\mu},\mathbf{\nu})=\left[r(1,\mu(1),\nu(1)),
\ldots, r(S,\mu(S),\nu(S)) \right]\in \mathbb{R}^{S}.$

\subsection{Solution Concept}
We now briefly discuss the question: ``What is a reasonable solution
concept for stochastic games?'' Recall that for zero-sum normal form
games, a Nash equilibrium is a widely used notion of equilibrium in
competitive scenarios. A Nash equilibrium in a two-player game is a
pair of independent randomized strategies (say $\mu$ and $\nu$, one
for each player) such that, given player $2$ plays the $\nu$, player
$1$'s best response would be to play $\mu$ and vice-versa. It is an
easy exercise that computation of Nash equilibria is equivalent to
finding saddle points of the payoff-function. It is also well-known
that Nash equilibria (or equivalently saddle points) correspond to the
minimax notion of an equilibrium, i.e. points that satisfy the
following equality:
$$
\min_{\mu} \max_{\nu} v(\mu, \nu) = \max_{\nu} \min_{\mu} v(\mu,
\nu).
$$ While there may exist no pure strategies that satisfy this
equality, it may be achieved by allowing randomization over the
allowable strategies. \\ \indent In his seminal paper \cite{Shapley},
Shapley generalized the notion of Nash equilibria to stochastic
games. He defined the notion of a ``stationary equilibrium'' to be a
pair of randomized strategies (over the action space) that depended
only on the state of the game. (Of course, to be an equilibrium, these
mixed strategies must also satisfy the no-deviation principle). For
stochastic games, once one restricts attention to stationary
equilibria, instead of having unique ``values'' (as in normal form
games), one has a unique ``value vector''. This vector is indexed by
the state and the $i^{th}$ component is interpreted as the equilibrium
value Player $1$ can expect to receive (over the infinite discounted
process) conditioned on the fact that the game starts in state
$i$. Note that different states of the game may be favorable to
different players. Since the actions affect both payoffs and state
transitions, players must balance their strategies so that they
receive good payoffs in a particular state along with favorable state
transitions. The ``no unilateral deviation" principle, saddle point
inequality (interpreted row-wise, i.e., conditioned upon a particular
state) and the equivalence of the minmax and maxmin over randomized
strategies all extend to the stochastic game case, and when we
restrict attention to games with just one state, we recover the
classical notions of equilibrium.

\begin{definition}
A pair of vector of mixed strategies (indexed by the state)
$\mathbf{\mu}^{0}$ and $\mathbf{\nu}^{0}$ which satisfy the saddle
point property:
\begin{equation}
\mathbf{v}_{\beta}(\mathbf{\mu},\mathbf{\nu}^{0}) \leq
\mathbf{v}_{\beta}(\mathbf{\mu}^{0},\mathbf{\nu}^{0}) \leq
\mathbf{v}_{\beta}(\mathbf{\mu}^{0},\mathbf{\nu})
\label{eq:saddlepoint}
\end{equation}
for all (vectors of) mixed strategies  $\mathbf{\mu}, \mathbf{\nu}$
are called \emph{equilibrium strategies}. The corresponding vector
$\mathbf{v}_{\beta}(\mathbf{\mu}^{0},\mathbf{\nu}^{0})$ is called
the \emph{value vector} of the game.
\end{definition}

One should note that $\mathbf{v}_{\beta}(\mu, \nu)$ is a vector in
$\mathbb{R}^{S}$ indexed by the initial state of the Markov process.
Hence the above inequality is a vector inequality and is to be
interpreted componentwise. More precisely, if $\mc{A}$ is the action
space, let $\Delta(\mc{A})$ denote the space of probability measures
supported on $\mc{A}$. Then the function $v_{\beta}$ is a function
of the form:
$$
v_{\beta}: \Pi_{i=1}^{S}\Delta(\mc{A}) \times \Pi_{i=1}^{S}\Delta(\mc{A}) \rightarrow \mathbb{R}^{S},
$$
and equilibrium strategies correspond to the saddle-points of this
function. The mixed strategies of the players are indexed by the
state (i.e. there is one probability measure per state per player).
These probability measures (conditioned upon the state) are
independent across states, and are also independent across the
players.

\subsection{Existence of Equilibria}
\label{sec:existence}

In his original paper, Shapley \cite{Shapley} showed that stationary
equilibria always exist (and that the corresponding value-vectors are
unique) for two-player, zero-sum, finite state, finite action
stochastic games. (Shapley considered games where at each state there
was some probability of termination, where as in this paper we
consider games over an infinite horizon with discounted rewards, as
already mentioned. These two formulations are equivalent in the sense
that starting from a discounted game one can construct a game with
termination probabilities and vice-versa such that both have the same
equilibrium value vectors.) In this subsection we address the
existence and uniqueness issue, and prove that for two-player,
zero-sum stochastic games over finite state spaces, infinite strategy
spaces, and polynomial payoffs, stationary equilibria always exist,
and that the value vectors are unique.  Throughout the paper, we
assume that the transition probabilities are polynomial functions of
the actions of the players. It is important to note that the results
of this subsection \emph{do not depend upon the single-controller
assumption}. As a by-product of this proof, we obtain a simple
algorithm for computing equilibria for all such games. This algorithm
is analogous to \emph{policy-iteration} in dynamic programming, and
consists of solving a sequence of simple (non-stochastic) games whose
value-vectors converge to the true value vector. \\ \indent Let
$p(x,y)$ be a polynomial, and $A=[0,1]$ be the strategy space of
players $1$ and $2$. Let $\val(p(x,y))$ be the value of the zero-sum
polynomial game with the payoff function as $p(x,y)$ and the strategy
space $A$. It can be shown that a mixed-strategy Nash equilibrium
always exists for two-player zero-sum polynomial games \cite{DKS}, and
they can be computed using semidefinite programming \cite{Par}.

\begin{lemma}
\label{lem:bound}
Let $p_{1}(x,y)$ and $p_{2}(x,y)$ be given polynomials. Then
$$
|\val(p_{1}(x,y))-\val (p_{2}(x,y) )| \leq \max_{x,y \in [0,1]}|p_{1}(x,y)-p_{2}(x,y)|.
$$
\end{lemma}
\begin{proof}
Let $\mu_1, \nu_1$ be the optimal strategies for the polynomial
zero-sum game with payoff $p_1(x,y)$ (so that
$\Ex_{\mu_1,\nu_1}[p_1(x,y)]=\val(p_{1}(x,y))$) and $\mu_2, \nu_2$ be
the optimal strategies for the game with payoff $p_{2}(x.y)$. If
$\val(p_1)=\val(p_2)$ the result is trivial, so without loss of
generality, assume that $\val(p_1)>\val(p_2)$. By the saddle point
property,
$$
\int p_1(x,y) d\mu_1 d\nu_2  \geq \int p_1(x,y) d\mu_1 d\nu_1 \geq \int p_2(x,y) d\mu_2 d\nu_2 \geq \int p_2(x,y) d\mu_1 d\nu_2.
$$ Here the first inequality follows by considering $\nu_2$ to be a
deviation of player $2$ from his optimal strategy (i.e. $\nu_1$) for
the game with payoff $p_1$, the second inequality follows by the
preceding assumption, and the third inequality follows from a
deviation argument for player $1$ from his optimal strategy. Hence,
$$
\begin{array}{cl}
\left|  \int p_1(x,y) d\mu_1 d\nu_1 - \int p_2(x,y) d\mu_2 d\nu_2 \right| & \leq \left|  \int (p_1(x,y) -p_2(x,y)) d\mu_1 d\nu_2 \right| \\
& \leq  \max_{x,y \in [0,1]} \left|   (p_1(x,y) -p_2(x,y))  \right|\int d\mu_1 d\nu_2.
\end{array}
$$
\end{proof}
Note that the quantity on the right is bounded because we are
considering the maximum of a bounded continuous function on a
compact set. Let $\alpha \in \mathbb{R}^{S}$. Given a polynomial
game with payoff functions $r(s,a_{1}, a_{2})$ and transition
probabilities $p(t;s,a_{1}, a_{2})$ (sometimes we will hide the
state indices and write the entire matrix as $P(a_1, a_2)$), fix a state $s$ and define the polynomial
$G^{s}(\alpha)=r(s,a_{1},a_{2})+\beta\sum_{t\in \mc{S}}p(t;s,a_{1},
a_{2})\alpha_t$. We will need to perform iterations using this vector $\alpha \in \mathbb{R}^{S}$.  We call the iterates of these vectors $\alpha^{k} \in \mathbb{R}^{S}$ ($k$ is the iteration index), and denote $s^{th}$ component of this vector by $\alpha^{k}_{s}$. Pick the vector $\alpha^{0} \in \mathbb{R}^{S}$ arbitrarily and define the
recursion for the $s^{th}$ component at iteration $k$ by:
$$
\alpha^{k}_{s}=\val(G^{s}(\alpha^{k-1})), \text{ \ \ \ \ \ } k=1,2, \ldots
$$ Rephrasing the above in terms of operators, define $T_s$ to be the
operator such that
$$T_{s}\alpha=\val(G^{s}(\alpha)).$$ 
Let $T \alpha=\left[ T_{1}\alpha,
\ldots T_{S}\alpha \right]^{T} $.Then the recursion simply consists of
computing the terms $T^{k}(\alpha)$.

\begin{lemma}
The quantity
$$
\lim_{k \rightarrow \infty} T^{k}(\alpha)=\phi
$$ 
exists and is independent of $\alpha$. Moreover, $\phi$ is the
unique fixed point solution to the equation:
$$
\phi=T\phi.
$$
\label{lem:fixedpoint}
\end{lemma}
\begin{proof}
For $\alpha \in \mathbb{R}^{S}$ define the norm $\| \alpha \|=\max_{s}|\alpha_{s}|$. Then,
$$
\begin{array}{cl}
\|T\gamma -T\alpha \|&=\max_{s}| \val(G^{s}(\gamma))-\val(G^{s}(\alpha))| \\
& \leq \max_{s} \max_{a_{1},a_2 \in[0,1]} | \beta \sum_{t} p(t;s,a_{1}, a_{2})(\gamma_t - \alpha_t)|  \text{ \ \ \ (using Lemma~\ref{lem:bound}) } \\
& \leq \max_{s} \max_{a_{1},a_2 \in[0,1]} | \beta \sum_{t} p(t;s,a_{1}, a_{2})| \max_{t} |(\gamma_t - \alpha_t)| \\
&=\beta \|\gamma -\alpha \|.
\end{array}
$$
Since the discount factor $\beta<1$, we have a contraction, and by the contraction mapping principle, the iteration $T^{k}\alpha$ is convergent to the unique fixed point of the equation $T\phi=\phi$.
\end{proof}
Lemma~\ref{lem:fixedpoint} establishes that a fixed point solution to
the iteration exists. We now show that the fixed point is in fact the
value vector of the game. To show this, we show that if we compute the
optimal strategies $\mu(s), \nu(s)$ to the game $G^{s}(\phi), s=1, 2,
\ldots, S$ then play according to these these strategies achieves the
value vector $\phi$. Since $\phi$ by definition satisfies the saddle
point inequality (1), an equilibrium solution exists. To show that the
value vector is unique, we show that any value vector satisfies the
fixed point equation $Tv_{\beta}=v_{\beta}$. Since there is a unique
fixed point by Lemma~\ref{lem:fixedpoint}, the value vector must be
unique.
\begin{theorem}
Let $\phi$ be the fixed point defined in
Lemma~\ref{lem:fixedpoint}. Then,
\begin{enumerate}

\item[a. ] Let $\mu(s), \nu(s)$ denote the optimal measures to the polynomial game with payoff $G^{s}(\phi)$, $s=\{1, \ldots, S \}.$ Then $\mu=[\mu(1), \ldots, \mu(S)]^{T}, \nu=[\nu(1), \ldots, \nu(S)]^{T}$  are the optimal strategies for the stochastic game.

\item[b. ] If $v_{\beta}(\mu,\nu)$ is a value vector for the game then $v_{\beta}$ satisfies $Tv_{\beta}=v_{\beta}$. Hence $v_{\beta}=\phi$ exists and is unique.

\end{enumerate}
\end{theorem}
\begin{proof}
Let $\mu(s)$ and $\nu(s)$ be the optimal strategies for the game
$G^{s}(\phi)$. Then by definition, the expected value of play under
these strategies will be $\phi_{s}=T_{s}\phi=
\ldots=T_{s}^{k}\phi$. Vectorizing this equation, we note that
$$
\phi=T^{k}\phi=\Ex_{\mu, \nu} [r(a_1, a_2)+\beta P(a_1,
a_2)r(a_1,a_2)+ \cdots+ \beta^{k-1}P^{k-1}(a_1, a_2)r(a_1,
a_2)+\beta^{k}P^{k}(a_1,a_2)\phi].
$$ 
Taking the limit as $k \rightarrow \infty$, we obtain that
$\phi=\Ex_{\mu, \nu}[\sum_{k=0}^{\infty} \beta^{k}P^{k}(a_{1},
a_2)r(a_1, a_2)]=v_{\beta}(\mu, \nu)$.  Hence playing according to the
stationary strategies $\mu(s), \nu(s), s=1, \ldots, S$ achieves the
value vector $\phi$.  Suppose player $1$ plays according to the
strategy $\mu$, and suppose player $2$ deviates from the prescribed
stationary strategy $\nu$ to stationary strategy $\nu \p$.  Then,
since $\mu, \nu$ are defined to be an equilibrium strategies for the
game $G^{s}(\phi)$, we have the (vector) inequality for all $\nu \p$:
$$
\begin{array}{cl}
\phi & =\Ex_{\mu, \nu}[r(a_1, a_2)+\beta P(a_1, a_2)\phi] \\
& \leq \Ex_{\mu, \nu\p}[r(a_1, a_2)+\beta  P(a_1, a_2) \phi] \\
& \leq \Ex_{\mu, \nu\p}[r(a_1, a_2)+\beta  P(a_1, a_2) r(a_1, a_2) + \beta^{2}  P^{2}(a_1, a_2) \phi ]\\
& \vdots \\
& \leq \Ex_{\mu, \nu\p}[r(a_1, a_2)+\beta  P(a_1, a_2) r(a_1, a_2) +
\cdots+ \beta^{k}  P^{k} (a_1, a_2)r(a_1, a_2) + \beta^{k}
P^{k}(a_1, a_2)\phi].
\end{array}
$$ 
In the first inequality a $\phi$ occurs on the right side. We
substitute that inequality in the $\phi$ on the right side to obtain
the second inequality and so on. Finally, we obtain the inequality:
$$
\phi=\Ex_{\mu, \nu}\left[\sum_{k=0}^{\infty} \beta^{k}P^{k}(a_1, a_2)r(a_1, a_2)\right]\leq \Ex_{\mu, \nu\p}\left[\sum_{k=0}^{\infty} \beta^{k}P^{k}(a_1, a_2)r(a_1, a_2) \right],
$$ 
i.e. that $\phi=v_{\beta}(\mu, \nu)\leq v_\beta(\mu, \nu \p)$ for
all $\nu\p$. A similar argument for deviations $\mu \p$ of player $1$
shows that $v_{\beta}(\mu \p, \nu) \leq v_{\beta}(\mu,
\nu)=\phi$. Hence $\mu(s), \nu(s)$ constructed as the strategies for
the games $G^{s}(\phi)$ satisfy the saddle point inequality
\eqref{eq:saddlepoint} component-wise. This establishes the existence
of equilibria. For uniqueness, note that any strategies $\mu, \nu$
such that $v_\beta(\mu, \nu)$ satisfies the saddle point inequality
\eqref{eq:saddlepoint}, by definition we have $Tv_{\beta}(\mu,
\nu)=v_{\beta}(\mu, \nu)$. Since $T$ has a unique fixed point, the
vector $v_{\beta}(\mu, \nu)$ must be unique.
\end{proof}

It is interesting to note that the above proof also provides an
algorithm to compute approximate equilibria. To compute each iterate
$T_{s}(\alpha)$ one needs to solve a polynomial game in normal form
(which can be done by solving a single semidefinite program), and by
solving a sequence of such problems, one can compute $T^{k}(\alpha)$
which is provably close to the actual value-vector. However, the
rate of convergence of this iteration is not very attractive. In the
rest of this paper, we focus attention on \emph{single-controller
games}, for which equilibria can be computed by solving a single
semidefinite program.

\subsection{SDP Characterization of Nonnegativity and Moments}
Let $A$ be a closed interval on the real line. The set of univariate
polynomials which are nonnegative on $A$ have an exact semidefinite
description. The set of (finite) vectors in $\mathbb{R}^{n}$ which
correspond to moment sequences of measures supported on $A$ also
have an exact semidefinite description. We briefly review
these notions here and introduce some related notation \cite{Par}.\\
\indent Let $\mathbb{R}[x]$ denote the set of univariate polynomials
with real coefficients. Let $p(x)=\sum_{k=0}^{n}p_{k}x^{k}\in
\mathbb{R}[x]$. We say that $p(x)$ is nonnegative on $A$ if
$p(x)\geq 0$ for every $x\in A$. We denote the set of nonnegative
polynomials of degree $n$ which are nonnegative on $A$ by
$\mc{P}(A)$. (To avoid cumbersome notation, we exclude the degree
information in the notation. Moreover the degree will usually be
clear from the context.) The polynomial $p(x)$ is said to be a
\emph{sum of squares} if there exist polynomials $q_{1}(x),\ldots,
q_{k}(x)$ such that $p(x)=\sum_{i=1}^{k}q_i(x)^{2}$. It is well known
that a
univariate polynomial is a sum of squares if and only if $p(x)\in \mc{P}(\mathbb{R})$.\\
 \indent Let $\mu$ denote a measure supported on the set $A$.
The $i^{th}$ moment of the measure $\mu$ is denoted by
$$\mu_{i}=\int_{A}x^{i}d\mu.$$
 Let $\bar{\mu}=[\mu_{0}, \ldots, \mu_{n}]$ be a vector in
$\mathbb{R}^{n+1}$. We say that $\bar{\mu}$ is a \emph{moment
sequence} of length $n+1$ if it corresponds to the first $n+1$
moments of some nonnegative measure $\mu$ supported on the set $A$.
The \emph{moment space}, denoted by $\mc{M}(A)$ is the subset of
$\mathbb{R}^{n+1}$ which corresponds to moments of nonnegative
measures supported on the set A. We say that a nonnegative measure
$\mu$ is a \emph{probability measure} if its zeroth order moment
satisfies $\mu_{0}=1$. The set of moment sequences of length $n+1$
corresponding to
probability measures is denoted by $\mc{M}_{P}(A)$.\\
\indent Let $\mc{S}^{n}$ denote the set of $n \times n$ symmetric
matrices and define the linear operator $\mc{H}:
\mathbb{R}^{2n-1}\rightarrow \mc{S}^{n}$ as:
$$
\mc{H}: \left[ \begin{array}{c} a_{1} \\ a_{2} \\ \vdots \\ a_{2n-1}
\end{array} \right] \mapsto \left[ \begin{array}{cccc} a_{1} &
a_{2}&
\ldots & a_{n} \\
a_{2} & a_{3} & \ldots & a_{n+1} \\
\vdots & \vdots & \ddots & \vdots \\
a_{n}& a_{n+1} & \ldots & a_{2n-1} \\
\end{array} \right].
$$
Thus $\mc{H}$ is simply the linear operator that takes a vector and
constructs the associated Hankel matrix which is constant along the
antidiagonals. We will also frequently use the adjoint of this
operator, the linear map $\mc{H}^{*}: \mc{S}^{n}\rightarrow
\mathbb{R}^{2n-1}$:
$$
\mc{H}^{*}: \left[ \begin{array}{cccc} m_{11} & m_{12}&
\ldots & m_{1n} \\
m_{12} & m_{22} & \ldots & m_{2n} \\
\vdots & \vdots & \ddots & \vdots \\
m_{1n}& m_{2n} & \ldots & m_{nn} \\
\end{array} \right] \mapsto \left[ \begin{array}{c} m_{11} \\ 2m_{12} \\ m_{22}+2m_{13} \\ \vdots \\
m_{nn}
\end{array} \right].
$$
This map flattens a matrix into a vector by adding all the entries
along antidiagonals.
\begin{lemma}
\label{lem:sos}
Let $p(x)=\sum_{k=0}^{2n}p_{k}x^{k}$ be a polynomial. Let
$\bar{p}=[p_{0}, \ldots, p_{2n}]^{T}$ be the vector of its
coefficients. Then $p(x)$ is nonnegative (or SOS) if and only if
there exists $S\in \mc{S}^{n+1}$, $S\succeq 0$ such that:
$$
\bar{p}=\mc{H}^{*}(S).
$$
\end{lemma}
\begin{proof}
For univariate polynomials, nonnegativity is equivalent to SOS (see
\cite{Par2}). Let $[x]_{n}=[1, x, \ldots, x^{n}]^{T}.$ We have for
every $S\in \mc{S}^{n+1}$,
$$p(x)=\bar{p}^{T}[x]_{2n}=\mc{H}^{*}(S)^{T}[x]_{2n}=[x]_{n}^{T}S[x]_{n}.$$
Factoring $S\succeq 0$, we obtain a sum of squares decomposition.
The converse is immediate.
\end{proof}
One can give a similar semidefinite characterization of polynomials
that are nonnegative on an interval. Since in this paper we are
typically considering the interval to be $[0,1]$ we give an explicit
semidefinite characterization of $\mc{P}([0,1])$. We define the
following matrices:
\begin{equation*}
\begin{array}{cc}
L_{1}=\left[ \begin{array}{c} I_{n \times n} \\
0_{1 \times n} \end{array} \right], & L_{2}=\left[ \begin{array}{c}
0_{1 \times n} \\ I_{n \times n} \end{array} \right],
\end{array}
\end{equation*}
where $I_{n \times n}$ stands for the $n \times n$ identity matrix.
\begin{lemma}
\label{lem:sos01}
The polynomial $p(x)=\sum_{k=0}^{2n}p_{k}x^{k}$ is nonnegative on
$[0,1]$ if and only if there exist matrices $Z\in \mc{S}^{n+1}$ and
$W\in \mc{S}^{n}$, $Z\succeq 0, W\succeq 0$ such that
$$ \left[ \begin{array}{c} p_{0} \\  \vdots\\ p_{2n}
\end{array}
\right]=\mc{H}^{*}(Z+\frac{1}{2}(L_{1}WL_{2}^{T}+L_{2}WL_{1}^{T})-L_{2}WL_{2}^{T}).
$$
\end{lemma}
\begin{proof}
The proof follows from the characterization of nonnegative
polynomials on intervals. It is well known that
$$
p(x)\geq 0 \text{ \ \ } \forall x \in [0,1] \Leftrightarrow
p(x)=z(x)+x(1-x)w(x),
$$
where $z(x)$ and $w(x)$ are sums of squares. A simple application of
Lemma~\ref{lem:sos} yields the required condition.
\end{proof}

In this paper, we will also be using a very important classical
result about the semidefinite representation of moment spaces
\cite{Karlin,Shohat}. We give an explicit characterization of
$\mc{M}([0,1])$ and $\mc{M}_{P}([0,1])$.
\begin{lemma}
The vector $\bar{\mu}=[\mu_{0}, \mu_{1}, \ldots, \mu_{2n}]^{T}$ is a
valid set of moments for a nonnegative measure supported on $[0,1]$
if and only if
\begin{equation}
\begin{aligned}
\mc{H}(\bar{\mu})& \succeq  0 \\
\frac{1}{2}(L_{1}^{T}\mc{H}(\bar{\mu})L_{2}+L_{2}^{T}\mc{H}(\bar{\mu})L_{1})-L_{2}^{T}\mc{H}(\bar{\mu})L_{2}& \succeq 0.
\end{aligned}
\label{eq:mom01}
\end{equation}
Moreover, it is a moment sequence corresponding to a probability
measure if and only if in addition to~\eqref{eq:mom01} it satisfies
$\mu_{0}=1$.
\label{lem:mom01}
\end{lemma}
\begin{proof}
The proof follows by dualizing Lemma~\ref{lem:sos01}. Alternatively, a
direct proof may be found in \cite{Karlin}.
\end{proof}

For example, for $2n=2$ the sequence $\left[ \mu_{0}, \mu_1, \mu_2
\right]$ is a moment sequence corresponding to a measure supported
on $[0,1]$ if and only if the following inequalities are true:
\begin{align*}
\left[ \begin{array}{cc} \mu_0 & \mu_1 \\ \mu_1 & \mu_2 \end{array}
\right] &\succeq 0 \\
\mu_1-\mu_2 &\geq 0.
\end{align*}

%%%%%%%%%%%%%%%%%%%%%%%%%%%%%%%%%%%%%%%%%%%%%%%%%%%%%%%%%%%%%%%%%%%%%%%%%%%%%%%%
\section{Finite Strategy Case}
\label{sec:finite}

For the reader's convenience and comparison purposes, we briefly
review here the case where each player has only finitely many
strategies at each state \cite{Koos}. Again, for simplicity we
assume that the set of pure strategies available to each player at
each state is identical so that $A_{1}=A_{2}=\{1,\ldots,m\}$. Under
the finite strategy case, when assumption $SC$ holds, a minimax
solution may be computed via linear programming. We state the linear
program in this section. In the next section, drawing motivation
from this linear program, we write an infinite dimensional
optimization problem for the case where each player has a choice
from infinitely many pure strategies.
 The finite action game
is completely defined via the specification of the following data:
\begin{enumerate}
\item The state space $\mathcal{S}=\{1,\ldots,S\}$.
\item The (finite) sets of actions for players $1$ and $2$ given by
$A_{1}=A_{2}=\{1,\ldots,m\}$.
\item The payoff function for a given state $s$ (representable by a matrix indexed by the actions of each players)
denoted by $r(s,a_{1},a_{2})$.
\item The probability transition matrix $p(s^{\prime};s,a_{1})$ which provides the conditional probability of
transition from state $s$ to $s^{\prime}$ given player $1$'s action
$a_{1}$.
\item The discount factor $\beta$.
\end{enumerate}
A \emph{mixed} strategy for player $1$ is a function
$f:\mathcal{S}\times A_{1} \rightarrow [0,1]$ subject to the
normalization constraint $\sum_{a_{1}}f(s,a_{1})=1$ for each
$s\in\mathcal{S}$ (so that $f(s)=\left[ f(s,1),\ldots,f(s,m)
\right]$ becomes a probability distribution over the strategy space
$A_{1})$. Similarly the mixed strategy for player $2$ in a
particular state $s$ is given by $g(s)=\left[g(s,1),\ldots,g(s,m)
\right]$. The collection of mixed strategies (indexed by the states)
will be denoted by $\mathbf{f}=\left[f(1),\ldots,f(S)\right]$ (and
$\mathbf{g}=\left[g(1),\ldots,g(S)\right]$ respectively). A strategy
$\mathbf{f}$ leads to a probability matrix
$P(\mathbf{f})=\sum_{a_{1}\in
A_{1}}p(s^{\prime};s,a_{1})f(s,a_{1})$. Again we consider a
$\beta$-discounted process over an infinite horizon. Given
strategies $\mathbf{f}$ and $\mathbf{g}$, the reward collected by
player $1$ in some stage $s$ is given by:
$$r(s,f(s),g(s))=\sum_{a_{1}\in A_{1},a_{2}\in
A_{2}}r(s,a_{1},a_{2})f(s,a_{1})g(s,a_{2}).$$ The reward collected
over the infinite horizon starting at state $s$,
$v_{\beta}(s,f(s),g(s))$, is given by the system of equations:
\begin{equation*}
\begin{array}{l}
v_{\beta}(s,f(s),g(s))=r(s,f(s),g(s))+ \\ \text{ \
}\beta\sum_{s\p\in \mc{S}} \left( \sum_{a_{1}\in A_{1}}
p(s^{\prime};s,a_{1})f(s,a_{1})\right) v_{\beta}(s^{\prime},f(s \p),
g(s \p)).% \text{\indent} \forall s.
\end{array}
\end{equation*}
Thus,
$$\mathbf{v}_{\beta}(\mathbf{f},\mathbf{g})=(I-\beta P(\mathbf{f}))^{-1}\mathbf{r}(\mathbf{f},\mathbf{g}),$$
where $\mathbf{r}(\mathbf{f},\mathbf{g})=\left[r(1,f(1),g(1)),
\ldots, r(S,f(S),g(S)) \right]\in \mathbb{R}^{S}.$ The problem is to
find equilibrium strategies $\mathbf{f}^{0}$ and $\mathbf{g}^{0}$
that satisfy the Nash equilibrium property:
\begin{equation} \label{eq:3}
\mathbf{v}_{\beta}(\mathbf{f},\mathbf{g}^{0}) \leq
\mathbf{v}_{\beta}(\mathbf{f}^{0},\mathbf{g}^{0}) \leq
\mathbf{v}_{\beta}(\mathbf{f}^{0},\mathbf{g})
\end{equation}
for all mixed strategies $\mathbf{f}, \mathbf{g}$.

\begin{theorem}[\cite{Koos}]
Consider the primal-dual pair of linear programs:
\begin{equation}
\begin{array}{ll}
& \text{\textrm{minimize}     }\sum_{s=1}^{S} v(s) \\
& g(s,a_{2}),v(s)\\
& \\
& v(s) \geq \sum_{a_{2}\in A_{2}}r(s,a_{1},a_{2})g(s,a_{2})+ \\
& \text{\indent \indent
\indent}\beta\sum_{s^{\prime}=1}^{S}p(s^{\prime};s,a_{1})v(s^{\prime})
\text{\indent} \forall s\in\mathcal{S}, a_{1}\in A_{1} \\
& \\
& \sum_{a_{2}\in A_{2}}g(s,a_{2})=1 \text{\indent} \forall s\in\mathcal{S} \\
& \\
& g(s,a_{2})\geq 0 \text{\indent} \forall  s\in\mathcal{S}, a_{2}\in A_{2}.
\end{array}
\tag{$P$}
\end{equation}
and
\begin{equation}
\begin{array}{ll}
& \text{\textrm{maximize}     }\sum_{s=1}^{S} z(s) \\
& x(s,a_{1}),z(s) \\
& \\
& \sum_{s=1}^{S} \sum_{a_{1}\in A_{1}} [\delta(s,s^{\prime})-\beta
p(s^{\prime},s,a_{1})]x({s,a_{1}})=1 \text{ \ \ \ } \forall
s \p \in\mathcal{S}\\
& \\
%& \text{ \ \ \ } \forall
%s \p \in\mathcal{S} \\
& z(s)\leq \sum_{a_{1}\in A_{1}} x(s,a_{1})r(s,a_{1},a_{2}) \text{ \ \ \ } \forall s\in\mathcal{S}, a_{2}\in A_{2},\\
& \\
%&\text{ \ \ \ } \forall s\in\mathcal{S}, a_{2}\in A_{2},  \\
& x(s,a_{1})\geq 0, \text{\indent} \forall  s\in\mathcal{S},
a_{1}\in A_{1}.
\end{array}
\tag{$D$}
\end{equation}

Let $p^{*}$ be the optimal value of $(P)$, and $d^{*}$ be the
optimal value of $(D)$. Let $x^{*}(s,a_{1})$ be the optimal values
of the $x(s,a_{1})$ variables obtained in $(D)$. Let
$$ f^{*}(s,a_{1})=\frac{x^{*}(s,a_{1})}{\sum_{a_{1}}x^{*}(s,a_{1})}
$$ and $g^{*}(s,a_{2})$ be the distribution  obtained by the optimal
solution of $(P)$. Then the following statements hold:
\begin{enumerate}
\item $p^{*}=d^{*}$. 
\item Let $\mathbf{v}^{*}=[v^{*}(1),\ldots,v^{*}(S)]$ be the optimal
solution of $(P)$. Then
$\mathbf{v}^{*}=\mathbf{v}_{\beta}(\mathbf{f}^{*},\mathbf{g}^{*}).$
\item $\mathbf{v}_{\beta}(\mathbf{f}^{*},\mathbf{g}^{*})$ satisfies the
saddle-point inequality (\ref{eq:3}).
\end{enumerate}
\end{theorem}
\begin{remark}
Note that statement $2$ claims that the solution of the LP $(P)$
corresponds to the infinite horizon discounted reward obtained when
players $1$ and $2$ play according to the distributions
$\mathbf{f}^{*}$ and $\mathbf{g}^{*}$. Statement $3$ claims that
these distributions are in fact optimal for the two players in the
Nash equilibrium sense.
\end{remark}
\begin{proof}
See \cite[pp. 93]{Koos}.
\end{proof}
\begin{remark}
Note that the primal problem $(P)$ has a natural interpretation in
terms of \emph{security strategies}. Feasible vectors $\mathbf{v}$,
and $\mathbf{g}$ satisfy the first set of inequalities in $(P)$. The
inequalities can be interpreted to mean that using strategy
$\mathbf{g}$ the payoff of player $2$ will be at most $\mathbf{v}$.
\end{remark}

%%%%%%%%%%%%%%%%%%%%%%%%%%%%%%%%%%%%%%%%%%%%%%%%%%%%%%%%%%%%%%%%%%%%%%%%%%%%%%%%
\section{Infinite Strategy Case}
\label{sec:infinite}

\subsection{Problem Setup}
In this section we consider the case where each player can choose from
uncountably many different actions. In particular, each player can
choose actions from the set $[0,1]$. The number of states
$|\mathcal{S}|=S$ is still finite. The payoff function
$r(s,a_{1},a_{2})$ is a polynomial in $a_{1}$ and $a_{2}$ for each
$s\in \mathcal{S}$. The single controller case (Assumption SC) is
studied. In this case, we assume that the probability of transition
$p(s\p;s,a_{1})$ is a polynomial in $a_{1}$. Again we consider the
two-player zero sum case where player $1$ attempts to maximize his
reward over the infinite horizon. We generalize the problem $(P)$ to
this case. The variables $\mathbf{f}$ and $\mathbf{g}$ representing
distributions over the finite sets $A_{1}$ and $A_{2}$ are replaced by
measures $\mu(s)$ and $\nu(s)$. These measures represent mixed
strategies over the uncountable action spaces. (We remind the reader
that for each player there are $S$ measures, each measure
corresponding to a mixed strategy in a particular state. For example
$\mu(s)$ corresponds to the mixed strategy player $1$ would adopt when
the game is in state $s$.)

\subsection{Preliminary Results}
We point out that the generalization of $(P)$ to this case is an
optimization problem involving non-negativity of a system of
univariate polynomials with coefficients that depend on the moments of
these measures. The interpretation in terms of security strategies for
player $2$ holds.  The following is the generalization of the linear
program $(P)$ mentioned above:
\begin{equation*}
\begin{array}{ll}
& \text{minimize       }\sum_{s=1}^{S} v(s) \\
& \nu(s),v(s) \\ \\
(a) &   v(s) \geq \int_{a_{2} \in A_{2}}
r(s,a_{1},a_{2})d\nu(s)+\\
& \text{\indent \indent
\indent}\beta\sum_{s^{\prime}=1}^{S}p(s^{\prime};s,a_{1})v(s^{\prime})
\text{ for all } s\in\mathcal{S}, a_{1}\in A_{1} \\ 
(b) & \nu(s) \text{ is a measure supported on } A_{2} \text{ for all } s\in \mathcal{S}  \\
\end{array}
\end{equation*}
Since $\int r(s,a_{1},a_{2})d\nu(s)=q_{\nu}(s,a_{1})$, a univariate
polynomial in $a_{1}$ for each $s\in \mathcal{S}$, for a fixed
vector $v({s})$, the constraints (a) are a system of polynomial
inequalities. Note that the coefficients of $q$ will depend on the
measure $\nu$ only via finitely many moments. More concretely, let
$r(s,a_{1},a_{2})=\sum_{i,j}^{n_{s},m_{s}}
r_{ij}(s)a_{1}^{i}a_{2}^{j}$ be the payoff polynomial. Then $\int
r(s,a_{1},a_{2})d\nu(s)=\sum_{i,j} r_{ij}(s)a_{1}^{i}\nu_{j}(s)$.
Using this observation, this problem may be rewritten as the
following problem.
%Let $\mathcal{P}_{n}(A_{1})$ denote the
%set of univariate polynomials of degree $n$ that are nonnegative on
%the set $A_{1}$. Let $\mathcal{M}_{m}(A_{2})$ denote the set of
%moments of a probability measure supported on the set $A_{2}$. Let
%$\bar{\nu}(s)=[\nu_{1}(s), \ldots, \nu_{m}(s)]$ be the vector of
%(first $m$) moments of the measure $\nu(s)$.
\begin{equation}
\begin{array}{ll}
& \text{minimize     } \sum_{s=1}^{S} v(s) \\
& \bar{\nu}(s),v(s) \\ \\
(c) & \text{\indent} v(s)- \sum_{i,j} r_{ij}(s)a_{1}^{i}\nu_{j}(s)-\\ \\
& \text{\indent}
\beta\sum_{s^{\prime}=1}^{S}p(s^{\prime};s,a_{1})v(s^{\prime}) \in
\mc{P}(A_{1}) \text{ for all } s\in\mathcal{S} \\ \\
(d) & \text{\indent}\bar{\nu}(s) \in \mathcal{M}(A_{2}), \text{ and } \nu_{0}(s)=1  \text{ for all } s \in \mathcal{S}. \\
\end{array}
\tag{$P \p$}
\end{equation}
The constraints (c) give a system of polynomial inequalities in
$a_{1}$, one inequality per state. Fix some state $s$. Let the
degree of the inequality for that state by $d_{s}$. Let $
[a_{1}]_{d_{s}}=[1, a_{1}, a_{1}^{2}, \ldots a_{1}^{d_{s}}]$. The
first term in constraint (c) can be rewritten in vector form as:
$$
\sum_{i,j}
r_{ij}(s)a_{1}^{i}\nu_{j}(s)=\bar{\nu}(s)^{T}R(s)^{T}[a_{1}]_{d_{s}},
$$
where $R(s)$ is a matrix that contains the coefficients of the
polynomial $r(s,a_{1},a_{2})$. Similar to the finite strategy case
we define a vector by
$\mathbf{v}^{*}=[v^{*}(1),\ldots,v^{*}(S)]^{T}$ which will turn out
to be the value vector of the stochastic game (which is indexed by
the state). The second term in the constraint (c) which depends on
the probability transition $p(s\p;s,a_{1})$ is also a polynomial in
$a_{1}$ whose coefficients depend on the coefficients of
$p(s\p;s,a_{1})$ and $\mathbf{v}$. Specifically
$$
\sum_{s^{\prime}=1}^{S}p(s^{\prime};s,a_{1})v(s^{\prime})=\mathbf{v}^{T}Q(s)^{T}[a_{1}]_{d_{s}},
$$
for some matrix $Q(s)$ which contains the coefficients of
$p(s\p;s,a_{1})$.
\begin{lemma} % Lemma 1
Let $A_{1}=A_{2}=[0,1]$. Let $E_{s}\in\mathbb{R}^{d_{s}\times S}$ be
the matrix which has a $1$ in the $(1,s)$ position. Then the
semidefinite program $(SP)$ given by:
\begin{equation}
\begin{array}{ll}
& \text{{minimize}    } \sum_{s=1}^{S} v(s) \\
& \bar{\nu}(s),v(s) \\ \\
(e)&
\mc{H}^{*}(Z_{s}+\frac{1}{2}(L_{1}W_{s}L_{2}^{T}+L_{2}W_{s}L_{1}^{T})-L_{2}W_{s}L_{2}^{T})\\
& =E_{s}\mathbf{v}-\beta Q(s) \mathbf{v} -R(s)\bar{\nu}(s) \text{ \
\ \ }\forall s \in \mc{S}\\ \\
(f)& \mc{H}(\bar{\nu}(s)) \succeq 0 \text{ \ \ \ }\forall s \in
\mc{S}\\ \\
(g)& \frac{1}{2}\left( L_{1}\t\mh L_{2} +L_{2}^{T} \mh L_{1}\right)\\
&-L_{2}\t \mh L_{2} \succeq 0 \text{ \ \ \ }\forall s \in \mc{S}\\
\\
(h)& e_{1} \t \bar{\nu}(s)=1  \text{ \ \ \ }\forall s \in
\mc{S}\\ \\
(i)& Z_{s}, W_{s} \succeq 0  \text{ \ \ \ }\forall s \in \mc{S}\\ \\
\end{array}
\tag{$SP$}
\end{equation}
exactly solves the polynomial optimization problem $(P \p)$.
\end{lemma}
\begin{proof}
The polynomial in inequality (c) has the coefficient vector
$E_{s}\mathbf{v}-\beta Q(s) \mathbf{v} -R(s)\bar{\nu}(s)$. The proof
follows as a direct consequence of Lemma~\ref{lem:sos01} concerning
the semidefinite representation of polynomials nonnegative over
$[0,1]$, and Lemma~\ref{lem:mom01} concerning the semidefinite
representation of moment sequences of nonnegative measures supported
on $[0,1]$.
\end{proof}

The dual of $(SP)$ is given by the following semidefinite program:
\begin{equation}
\begin{array}{ll}
& \text{maximize    } \sum_{s=1}^{S} \alpha(s) \\
& \alpha(s),\bar{\xi}(s) \\ \\
(j)&
\mc{H}^{*}(A_{s}+\frac{1}{2}(L_{1}B_{s}L_{2}^{T}+L_{2}B_{s}L_{1}^{T})-L_{2}B_{s}L_{2}^{T})=\\
& \text{\indent \indent \indent} R_{s}^{T}\xib(s)- \alpha(s)e_{1}
\text{ \ \ \ }\forall s \in
\mc{S}\\ \\
(k)& \mc{H}(\xib(s)) \succeq 0 \text{ \ \ \ }\forall s \in \mc{S}\\ \\
(l)& \frac{1}{2}\left( L_{1}\t\mxi L_{2} +L_{2}^{T} \mxi
L_{1}\right)- \\
& \text{\indent \indent \indent} L_{2}\t \mxi L_{2} \succeq 0 \text{ \ \ \ }\forall s \in \mc{S}\\
\\
& \sum_{s} (E_{s}- \beta Q(s))\t \xib(s)=1 \\ \\
(m)& A_{s}, B_{s} \succeq 0  \text{ \ \ \ }\forall s \in \mc{S}.\\ \\

\end{array}
\tag{$SD$}
\end{equation}

\begin{lemma}
The dual SDP $(SD)$ is equivalent to the following polynomial
optimization problem:
\begin{equation}
\begin{array}{ll}
& \text{{maximize}    } \sum_{s=1}^{S} \alpha(s) \\
& \alpha(s),\bar{\xi}(s) \\ \\
(n) & \sum_{i,j}r_{ij}(s)\xi_{i}(s)a_{2}^{j}- \alpha(s) \geq 0
\text{
\indent }\forall a_{2} \in A_{2}, s \in \mc{S}\\ \\
(o) &  \xib(s) \in \mc{M}(A_{2}) \text{ \ \ \ }\forall s \in \mc{S}\\
\\
(p) & \sum_{s} \int_{A_{1}} (\delta(s, s \p) - \beta p(s \p, s,
a_{1}))
d\xi (s) =1 \text{ \ }\forall s \p \in \mc{S}.\\ \\
\end{array}
\tag{$D \p$}
\end{equation}
\end{lemma}
\begin{proof}
This again follows as a consequence of Lemmas~\ref{lem:sos01} and~\ref{lem:mom01}.
\end{proof}
\begin{remark}
 Note that in
the dual problem, the moment sequences do not necessarily correspond
to probability measures. Hence, to convert them to probability
measures, one needs to normalize the measure. Upon normalization,
one obtains the optimal strategy for player $1$.
\end{remark}

\begin{lemma}
The polynomial optimization problems $(P \p)$ and $(D \p)$ are
strong duals of each other.
\end{lemma}
\begin{proof}
We prove this by showing that the semidefinite program $(SP)$
satisfies Slater's constraint qualification and that it is bounded
from below. The result then follows from the strong duality of the
equivalent semidefinite programs $(SP)$ and
$(SD)$.\\
First pick $\mu(s)$ and $\nu(s)$ to be the uniform distribution on
$[0,1]$ for each state $s\in \mc{S}$. One can show \cite{Karlin}
that the moment sequence of $\mu$ is in the interior of the moment
space of $[0,1]$. As a consequence, constraints (f) and (g) are
strictly positive definite. Using the strategies $\mathbf{\mu}$ and
$\mathbf{\nu}$, evaluate the discounted value of this pair of
strategies as:
$$
\mathbf{v}_{\beta}(\mathbf{\mu},\mathbf{\nu})=[I-\beta
P(\mu)]^{-1}\mathbf{r}(\mu,\nu).
$$
Choose $\mathbf{v}>\mathbf{v}_{\beta}$. The polynomial inequalities
given by (c) are all strictly positive and thus constraints (i) are
strictly positive definite. The
equality constraints are trivially satisfied. \\
To prove that the problem is bounded below, we note that
$r(s,a_{1},a_{2})$ is a polynomial and that the strategy spaces for
both players are bounded. Hence,
$$\inf_{a_{1}\in A_{1}, a_{2}\in A_{2}} r(s,a_{1},a_{2})$$
is finite and provides a trivial lower bound for $v(s)$.
\end{proof}
\begin{lemma}
\label{lem:cs}
Let $\bar{\nu}^{*}(s)$ and $\bar{\xi}^{*}(s)$ be optimal moment
sequences for $(P \p)$ and $ (D \p)$ respectively. Let $\nu^{*}(s)$
and $\xi^{*}(s)$ be the corresponding measures supported on $A_{1}$
and $A_{2}$ respectively. The following complementary slackness
results hold for the optima of $(P \p)$ and $ (D \p)$:
\begin{equation}
\begin{array}{l}
v^{*}(s)\int_{A_{1}}d\xi^{*}(s)=\int_{A_{2}} \int_{A_{1}}
r(s,a_{1},a_{2})d\xi^{*}(s)d\nu^{*}(s) +\\
\text{\indent \indent \indent}\beta \sum_{s \p}
v^{*}(s\p)\int_{A_{1}} p(s \p; s,
a_{1})d\xi^{*}(s) \text{ \ \ \ }\forall s \in \mc{S}\\
\end{array} 
\label{eq:cs1}
\end{equation}

\begin{equation} \begin{array}{l}
\alpha^{*}(s) \int_{A_{2}}d\nu^{*}(s)=\int_{A_{2}} \int_{A_{1}}
r(s,a_{1}, a_{2}) d\xi^{*}(s)d\nu^{*}(s) \\
 \text{ \indent \indent \indent \indent \indent \indent \indent \indent \indent \indent
 \indent \indent \indent \indent \indent \indent \indent \indent }\forall s \in \mc{S}.
\end{array}
\label{eq:cs2}
\end{equation}
\end{lemma}
\begin{proof}
The result follows from the strong duality of the equivalent
semidefinite representations of the primal-dual pair $(P \p)-(D\p)$.
The Lagrangian function for $(P \p)$ is given by:
\begin{equation*}
\begin{array}{rl}
\mc{L}(\xi,\alpha)=&\inf_{\mathbf{v}, \nu} \{
\sum_{s=1}^{S}v(s)-\int_{A_{1}} [ v(s) -  \int_{A_{2}}
r(s,a_{1},a_{2})d\nu(s)  \\
& - \beta \sum_{s\p} v(s\p) p(s \p; s, a_{1}) ]d\xi(s)
+\sum_{s}\alpha(s)(1-\nu_{0}(s))\}.
\end{array}
\end{equation*}
$\mc{L}(\xi,\alpha)$ must satisfy weak duality, i.e. $d^{*}\leq
p^{*}$. At optimality $p^{*}=\sum_{s}v^{*}(s)$ for some vector
$\mathbf{v}^{*}$. However, strong duality holds, i.e. $p^{*}=d^{*}$.
This forces the first complementary slackness relation. The second
relation is obtained similarly by considering the Lagrangian of the
dual problem.
\end{proof}

We have shown that problem $(P \p)$ can be reduced to the
semidefinite program $(SP)$, and is thus computationally tractable
via convex optimization algorithms. We next show that the solution
to problem $(P \p)$ is in fact the desired equilibrium solution.
%Problem $(P \p)$ is computationally tractable via convex
%optimization. The constraints (c) are univariate polynomial
%positivity constraints. For univariate polynomials, it is well-known
%that polynomial positivity is equivalent to the polynomial being a
%sum of squares. Since the coefficients of the polynomials are linear
%in the decision variables, constraints (c) have a semidefinite
%representation. Constraints (d) require that the vectors $\bar{\nu}$
%represent a (truncated) sequence of moments of some measure
%supported on the compact basic semialgebraic set (described by
%univariate polynomial inequalities) $A_{2}$. It is well-known
%\cite{Karlin,Shohat,Lasserre} that the set of such sequences also
%have a semidefinte representation. These observations allowed us to
%recast $(P \p)$ as a finite dimensional convex optimization problem.
%We prove in the next section that the solution to the problem $(P
%\p)$ provides the optimal cost to go and an optimal strategy for
%player $2$ and that the solution to $(D \p)$ provides an optimal
%strategy for player $1$.

\subsection{Main Theorem}
Let $p^{*}$ be the optimal value of $(P \p)$, and $d^{*}$ be the
optimal value of $(D \p)$. Let $\nu^{*}(s)$ and $\xi^{*}(s)$ be the
optimal measures recovered in $(P \p)$ and $(D \p)$. Let
$$ \mu^{*}(s)=\frac{\xi^{*}(s)}{\int_{A_{1}}d \xi^{*}(s)}.
$$ so that $\mu^{*}$ is a normalized version of $\xi^{*}$ (i.e. $\mu^{*}$ is a probability measure).
Let $\mathbf{v}^{*}$ be the vector obtained as the optimal solution
of $(P \p)$.
\begin{theorem}
The optimal solutions to the primal-dual pair $(P \p)$, $(D \p)$
satisfy the following:
\begin{enumerate}
\item $p^{*}=d^{*}$.
\item $\mathbf{v}^{*}=v_{\beta}(\mathbf{\mu}^{*},\mathbf{\nu}^{*}).$
\item $\mathbf{v}_{\beta}(\mathbf{\mu}^{*},\mathbf{\nu}^{*})$ satisfies the
saddle-point inequality:
\begin{equation}
\mathbf{v}_{\beta}(\mathbf{\mu},\mathbf{\nu}^{*}) \leq
\mathbf{v}_{\beta}(\mathbf{\mu}^{*},\mathbf{\nu}^{*}) \leq
\mathbf{v}_{\beta}(\mathbf{\mu}^{*},\mathbf{\nu})
\end{equation}
for all mixed strategies $\mathbf{\mu}, \mathbf{\nu}$.
\end{enumerate}
\end{theorem}
\begin{proof}
\begin{enumerate}
\item Follows from the strong duality of the primal-dual pair $(P \p)-(D
\p)$.
\item Using Lemma~\ref{lem:cs} equation~\eqref{eq:cs1} in normalized
form (i.e. dividing throughout by $\xi_{0}^{*}(s)$, which is the
zeroth order moment of the measure $\xi(s)$) we obtain
\begin{equation*}
\begin{array}{ll}
v^{*}(s)=\int_{A_{2}} \int_{A_{1}}
r(s,a_{1},a_{2})d\mu^{*}(s)d\nu^{*}(s) + \\
\text{\indent \indent \indent }\beta \sum_{s \p} v^{*}(s
\p)\int_{A_{1}} p(s \p; s, a_{1})d\mu^{*}(s) \text{ \ \ \ }\forall s
\in \mc{S}.
\end{array}
\end{equation*}
Upon simplification and vectorization of $v^{*}(s)$ one obtains
$$
\mathbf{v}^{*}=r(\mathbf{\mu^{*}},\nu^{*}) + \beta
P(\mu^{*})\mathbf{v}^{*}.
%\sum_{s \p} v^{*}(s \p)\int_{A_{1}} p(s \p; s, a_{1})d\mu^{*}(s).
$$
Using a Bellman equation argument or by simply iterating this
equation (i.e. substituting repeatedly for $\mathbf{v}^{*}$) it is
easy to see that
$\mathbf{v}^{*}=\mathbf{v}_{\beta}(\mathbf{\mu^{*}},\mathbf{\nu^{*}})$.
\item Consider inequality (c) it at its optimal value.
We have for every state $s$:
$$
\begin{array}{c}
v^{*}(s) \geq \int_{a_{2} \in A_{2}}
r(s,a_{1},a_{2})d\nu^{*}(s)+\\
\text{ \ \ \ \ \ \ \ \ \ \ }
\beta\sum_{s^{\prime}=1}^{S}p(s^{\prime};s,a_{1})v^{*}(s^{\prime}).
\end{array}
$$
Integrating with respect to some arbitrary probability measure
$\mu(s)$ (with support on $A_{1}$), we get:
$$
\begin{array}{c}
v^{*}(s) \geq \int_{A_{2}} \int_{A_{1}} r(s,a_{1},a_{2})d\mu(s)
d\nu^{*}(s)+\\
\text{\ \ \ \ \ \ \ \ \ } \beta\sum_{s^{\prime}=1}^{S} \int_{A_{1}}
p(s^{\prime};s,a_{1})v^{*}(s^{\prime})d\mu(s).
\end{array}
$$
Thus,
\begin{equation*}
\begin{array}{l}
v^{*}(s) \geq r(s,\mu(s),\nu^{*}(s))+ \\
\text{\ \ \ \ \ \ \ \ \ }
 \beta\sum_{s^{\prime}=1}^{S} \int_{A_{1}}
p(s^{\prime};s,a_{1})v^{*}(s^{\prime})d\mu(s).
\end{array}
\end{equation*}
Iterating this equation, we obtain
$\mathbf{v}_{\beta}(\mathbf{\mu}^{*},
\mathbf{\nu}^{*})=\mathbf{v}^{*}\geq
\mathbf{v}_{\beta}(\mathbf{\mu},\mathbf{\nu^{*}})$ for every
strategy $\mathbf{\mu}$. This completes one side 
of the saddle point inequality.

Using the normalized version of equation (5), we get:
$$
\begin{array}{ll}
\frac{\alpha^{*}(s)}{\xi_{0}^{*}(s)} = & \int_{A_{2}} \int_{A_{1}}
r(s,a_{1}, a_{2}) d\mu^{*}(s)d\nu^{*}(s)
\\
& =r(s,\mu^{*}(s),\nu^{*}(s)).
\end{array}
$$
If we integrate inequality (n) in problem $(D \p)$ with respect to
any arbitrary probability measure $ \nu(s)$ with support on $A_{2}$
we obtain
$$
\frac{\alpha^{*}(s)}{\xi_{0}^{*}(s)} \leq r(s,\mu^{*}(s),\nu(s)).
$$
Thus $r(s,\mu^{*}(s),\nu^{*}(s))\leq r(s,\mu^{*}(s),\nu(s))$ for
every $s$. Multiplying throughout by $(I-\beta
P(\mathbf{\mu}^{*}))^{-1}$, we get
$\mathbf{v}_{\beta}(\mathbf{\mu}^{*},\mathbf{\nu}^{*}) \leq
\mathbf{v}_{\beta}(\mathbf{\mu}^{*},\mathbf{\nu}).$ This completes
the other side of the saddle point inequality.
\end{enumerate}
\end{proof}

\subsection{Obtaining the measures}
\label{sec:measures}

Solutions to the semidefinite programs $(SP)$ and $(SD)$ provide the
moment sequences corresponding to optimal strategies. Additional
computation is required to recover the actual measures. We briefly
describe a classical procedure to recover the measures using linear
algebra.
For more details, the reader may refer to \cite{Shohat}, \cite{Devroye}.

Let $\bar{\mu} \in \mathbb{R}^{2n}$ be a given moment sequence. We
wish to find a nonnegative measure $\mu$ supported on the real line
with these moments. The resulting measure will be composed of finitely
many atoms (i.e. a discrete measure) of the form $\sum w_{i}
\delta(x-a_{i})$ where
$$
\mathbf{Prob}(x=a_{i})=w_{i} \text{ \ \ } \forall i.
$$
Construct the following linear system: \\
\begin{equation*}
\left[ \begin{array}{cccc} \mu_{0} & \mu_{1} & \ldots & \mu_{n-1} \\
\mu_{1} & \mu_{2} & \ldots & \mu_{n} \\
\vdots & \vdots & \ddots & \vdots \\
\mu_{n-1} & \mu_{n} & \ldots & \mu_{2n-2} \end{array} \right] \left[
\begin{array}{c} c_{0} \\c_{1} \\ \vdots \\ c_{n-1} \end{array} \right]=
-\left[
\begin{array}{c} \mu_{n} \\\mu_{n+1} \\ \vdots \\ \mu_{2n-1} \end{array}
\right].
\end{equation*}

Note that the Hankel matrix that appears on the left hand side is a
sub-matrix of $\mc{H}(\bar{\mu})$. We assume without loss of
generality that the above matrix is strictly positive definite.
(Suppose the above matrix is not full rank, construct a smaller $k
\times k$ linear system of equations by eliminating the last $n-k$
rows and columns of the matrix so that the $k \times k$ submatrix is
full rank, and therefore strictly positive definite.) By inverting
this matrix we solve for $[c_{0}, \ldots, c_{n-1}]^{T}.$ Let $x_{i}$
be the roots of the polynomial equation
$$
x^{n}+c_{n-1}x^{n-1}+\cdots+c_{1}x+c_{0}=0.
$$
It can be shown that the $x_{i}$ are all real and distinct, and that
they are the support points of the discrete measure. Once the
supports are obtained, the weights $w_{i}$ may be obtained by
solving the nonsingular Vandermonde system given by:
\begin{equation*}
\sum_{i=1}^{n}w_{i}x_{i}^{j}=\mu_{j} \text{\ \ \ \ } (0 \leq j \leq
n-1).
\end{equation*}

%%%%%%%%%%%%%%%%%%%%%%%%%%%%%%%%%%%%%%%%%%%%%%%%%%%%%%%%%%%%%%%%%%%%%%%%%%%%%%%%
\section{Example}
\label{sec:example}

\begin{figure}
\begin{center}
\includegraphics{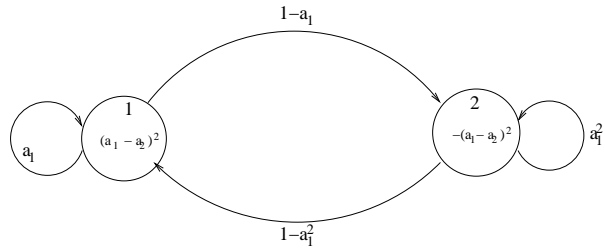}
\end{center}
\caption{A two state stochastic game with transition probabilities
dependent only on the action of Player 1. The payoffs associated to
the states are indicated in the corresponding nodes. The edges are
marked by the corresponding state transition probabilities. }
\label{fig:example}
\end{figure}

Consider the two player discounted stochastic game with $\beta=0.5$,
$\mathcal{S}=\{ 1,2 \}$ with payoff function
$r(1,a_{1},a_{2})=(a_{1}-a_{2})^{2}$ and
$r(2,a_{1},a_{2})=-(a_{1}-a_{2})^{2}$. Let the probability transition
matrix be given by: 
$$P(a_{1})=\left[ \begin{array}{cc} a_{1} &
1-a_{1} \\ 1-a_{1}^{2} & a_{1}^{2}
\end{array} \right].
$$ 
Figure~\ref{fig:example} graphically illustrates this stochastic
game, consisting of two states (the nodes) with polynomial transition
probabilities dependent on $a_{1}$ (as marked on the edges of the
graph). Within the nodes, the payoffs associated to the corresponding
states are indicated.

To understand this game, consider first the zero-sum (nonstochastic
game) with payoff function $p(a_1, a_2)=(a_1-a_2)^{2}$ over the
strategy space $[0,1]$. This game (called the ``guessing game'') was
studied by Parrilo in \cite{Par}. If Player $2$ is able to guess the
action of Player $1$, he can simply imitate his action (i.e. set
$a_2=a_1$ and his payoff to player $1$ would be zero (this is the
minimum possible since $(a_1-a_2)^{2} \geq 0$).  Player $1$ would try
to confuse player $2$ as much as possible and thus randomize between
the extreme actions $a_1=0$ and $a_1=1$ with a probability of
$\frac{1}{2}$. Player $2$'s best response would be to play
$a_2=\frac{1}{2}$ with probability $1$.

In the game described in Fig. 2, in State 1 Player $1$ plays the role
of confuser and Player $2$ plays the role of guesser. In state $2$,
the roles of the players are reversed, Player $1$ is the guesser and
Player $2$ the confuser. However, the problem is complicated a bit by
the fact that State 1 is advantageous to Player $1$ so that at every
stage he has incentive to play a strategy that gives him a good payoff
as well as maximize the chances of transitioning to State $1$.

The polynomial optimization problem that computes the minimax strategies
and the equilibrium values is the following:
\begin{equation*}
\begin{array}{l}
\text{minimize    } v(1)+v(2) \\ \\
 v(1) \geq \int (a_{1}-a_{2})^{2}d\nu(1)+\\
\text{\ \ \ \ \ \ \ \ \ \ } \beta(a_{1}v(1)+(1-a_{1})v(2)) \text{ \
\ } \forall a_{1} \in [0,1]\\ \\
 v(2) \geq -\int (a_{1}-a_{2})^{2}d\nu(2)+ \\
 \text{\ \ \ \ \ \ \ \ \ \ }\beta((1-a_{1}^{2})v(1)+a_{1}^{2}v(2))
\text{ \ \ } \forall a_{1} \in [0,1]\\ \\
 \nu(1), \nu(2) \text{ probability measures supported on } [0,1].
\\ \\
\end{array}
\end{equation*}

This problem can be reformulated as follows:
\begin{equation*}
\begin{array}{l}
\text{minimize    }  v(1)+v(2) \\ \\
 v(1) \geq a_{1}^{2}-2a_{1}\nu_{1}(1)+\nu_{2}(1)+\\
 \text{ \ \ \ \ \ \ \ }\beta(a_{1}v(1)+(1-a_{1})v(2)) \text{ \ \ }\forall a_{1} \in
 [0,1]\\ \\
 v(2) \geq -a_{1}^{2}+2a_{1}\nu_{1}(2)-\nu_{2}(2)+ \\
\text{ \ \ \ \ \ \ \ }\beta((1-a_{1}^{2})v(1)+a_{1}^{2}v(2)) \text{ \ \ } \forall a_{1} \in [0,1] \\ \\
\left[ 1, \nu_{1}(1), \nu_{2}(1) \right]^{T}, [1, \nu_{1}(2),
\nu_{2}(2)]^{T} \in \mc{M}([0,1]).
%& [ \begin{array}{ccc} 1 & \nu_{1}(1) & \nu_{2}(1) \end{array}],
%\left[ \begin{array}{ccc} 1 & \nu_{1}(2) & \nu_{2}(2)
\end{array}
\end{equation*}
Solving the SDP and its dual we obtain the following optimal
cost-to-go and optimal moment sequences:
$$ \begin{array}{c} \mathbf{v}^{*}=[ .298, -.158 ]^{T} \end{array}$$
$$
\begin{array}{cl}
\mathbf{\bar{\mu}}^{*}(1)=[1,.614,
.614]^{T} & \mathbf{\bar{\mu}}^{*}(2)=[1,.5,.25]^{T} \\
\mathbf{\bar{\nu}}^{*}(1)=[ 1, .614, .377]^{T} &
\mathbf{\bar{\nu}}^{*}(2)=[1, .614,.614]^{T}.
\end{array}
$$ 
The corresponding measures obtained as explained in
subsection~\ref{sec:measures} are supported at only finitely many
points, and are given by the following:
\begin{align*}
\mu^{*}(1)&=.386\text{ }\delta(a_{1})+.614\text{ }\delta(a_{1}-1) \\
\mu^{*}(2)&=\delta(a_{1}-.5) \\
\nu^{*}(1)&=\delta(a_{2}-.614) \\
\nu^{*}(2)&=.386\text{ }\delta(a_{2})+.614\text{ }\delta(a_{2}-1).
\end{align*}
Consider, for example, play in State $1$. If Player $1$ were playing
obliviously with respect to the state transitions, he would play
actions $a_1=0$ and $a_1=1$ with one half probability each. However,
to increase the probability of staying in State $1$ he plays action
$1$ with a higher probability. Player $2$ cannot affect the state
transition probabilities directly, thus he must play a myopic best
response. (A myopic best response is one that is a best response for
the game in the current state). Note that in state $1$, once Player
$1$'s strategy is fixed, the (only) best response for Player $2$ is to
play the action $a_2=0.614$ with probability $1$. In state $2$, player
$1$'s best strategy is to play $a_1=0.5$. Player $2$ picks an action
from his myopic best response set (in this case, all probability
distributions that are supported on the points $0$ and $1$).

\section{Conclusions and future work}
\label{sec:conclusions}

In this paper, we have presented a technique for solving two-player,
zero-sum finite state stochastic games with infinite strategies and
polynomial payoffs. We established the existence of equilibria for
such games. As a by-product we got an algorithm that converged to
unique value vector of the game (however this algorithm does not seem
to have very attractive convergence rates). We focused mainly on the
case where the single-controller assumption holds. We showed that the
problem can be reduced to solving a system of univariate polynomial
inequalities and moment constraints. We used techniques from the
classical theory of moments and sum-of-squares to reduce the problem
to a semidefinite programming problem. By solving a primal-dual pair
of semidefinite programs, we obtained minimax equilibria and optimal
strategies for the players.

It is known that finite-state, finite action, two-player zero-sum
games which satisfy the \emph{orderfield} property \cite{Raghavan},
\cite{Partha} may be solved via linear programming.  The
single-controller case, games with perfect information, switching
controller stochastic games, separable reward-state independent
transition (SER-SIT) games and additive games all satisfy this
property. We intend to extend these cases to the infinite strategy
case with polynomial payoffs. General finite action stochastic games
which do not satisfy the orderfield property still have an interesting
mathematical structure, but efficient computational procedures are not
available. Developing such procedures present an interesting direction
of future research.

\\

%{\center ACKNOWLEDGEMENT \\ }

\vspace{1cm}
\noindent\textbf{Acknowledgement:} The authors would like to thank Ilan
Lobel and Prof. Munther Dahleh for bringing to their attention the
linear programming solution to single controller finite stochastic
games.
%\end{acknowledgement}
\\ \\

\bibliographystyle{IEEEtran}

\bibliography{stochgames}

\end{document}